\documentclass[reqno, a4paper, 11pt]{amsart}
\usepackage[utf8]{inputenc}
\usepackage[T1]{fontenc}

\usepackage{amssymb, mathrsfs}
\usepackage{times}

\usepackage[colorlinks=true, bookmarks=true, pdfstartview=FitH, pagebackref=true, linktocpage=true]{hyperref}

\usepackage[short,nodayofweek]{datetime}
\usepackage[pagewise,running,mathlines,displaymath,switch]{lineno}

\usepackage{graphicx}
\usepackage{xcolor}
\usepackage{tabularx}
\usepackage{indentfirst}

\usepackage[utf8]{inputenc}  

\newcommand{\eq}[2]{\begin{equation}\begin{split}#1\end{split}#2\end{equation}}
\newcommand{\mat}[1]{\begin{pmatrix}#1\end{pmatrix}}

\usepackage{listings}
\usepackage[a4paper, vmargin={30mm,30mm},hmargin={20mm,20mm}, top=3.4cm, bottom=3.4cm, left=2.8cm, right=2.8cm]{geometry}

\theoremstyle{plain}
\numberwithin{equation}{section}
\newtheorem{theorem}{Theorem}[section] 
\newtheorem{lemma}{Lemma}[section] 

\newtheorem{proposition}{Proposition}[section]

\definecolor{brown}{rgb}{0.5,0,0}
\definecolor{backgroundcolor}{rgb}{0.98, 0.92, 0.73}

\def\R{\mathbb R}
\def\C{\mathbb C}
\def\bZ{\mathbb Z}

\def\cA{\mathcal A}
\def\cQ{\mathcal Q}

\def\bB{\mathcal B} 
\def\bT{\mathbb T} 
\newif\ifprint
\ifprint
\else
\fi

\author[B. Barker]{Blake Barker}
\address{Blake Barker: Department of Mathematics, Brigham Young University, Provo, UT 84602, USA }
\email{\href{mailto: B. Barker] <blake@math.byu.edu>}{blake@math.byu.edu}}

\author[T.-T. Nguyen]{Tien-Tai Nguyen}
\address{Tien-Tai Nguyen: University of Science, Vietnam National University, Hanoi, Vietnam}
\email{\href{mailto: T.-T. Nguyen <nttai.hus@vnu.edu.vn>}{nttai.hus@vnu.edu.vn}
}

\author[B.Settle]{Benjamin Settle}
\address{Benjamin Settle: Department of Mathematics, Brigham Young University, Provo, UT 84602, USA }
\email{\href{mailto: B. Settle] <bdsettle@student.byu.edu >}{bdsettle@student.byu.edu }}

\begin{document}
\allowdisplaybreaks

\setpagewiselinenumbers
\setlength\linenumbersep{100pt}

\title[Linear Rayleigh-Taylor instability]{Spectral analysis of the classical Rayleigh-Taylor instability with an upper free surface}

\begin{abstract}
In this note, we are interested in the linear Rayleigh- Taylor instability problem for the incompressible fluid with an upper free surface. Using the spectral theory of self-adjoint and compact operator, we rigorously prove the existence of infinitely many normal mode solutions to the linearized equations.  Some numerical computations are presented to support our theoretical study.
\end{abstract}

\date{\bf \today \,   at   \currenttime}

\subjclass[2010]{34B07, 47A10, 47B07, 76D05, 76E30}

\keywords{Rayleigh--Taylor instability, spectral analysis,  variational structure, Evans function}

\maketitle


\section*{Acknowledgements}
The second author thanks Wassim Aboussi for his constructive comments regarding the manuscript.  Part of this work was completed during the second author’s visit to the Institut Henri Poincaré (IHP) for the thematic program \textit{Mathematical Developments in Geophysical Fluid Dynamics} in April–May 2026. The second author gratefully acknowledges CIMPA for supporting his stay in Paris and thanks IHP for providing excellent working conditions.
\section{Introduction}

The Rayleigh--Taylor (RT) instability, first studied by Lord Rayleigh \cite{Str83} and later by  Taylor \cite{Tay50}, is well known as a gravity-driven instability that occurs when two inviscid, incompressible  fluids are vertically stacked with the heavier one on top of the lighter one. The RT instability  has attracted much attention due to both its physical and mathematical importance. Two notable applications are implosion of inertial confinement fusion capsules \cite{Lin98} and core-collapse of supernovae  \cite{Rem00}. For a detailed physical review of the RT instability, we refer to the survey articles \cite{Kull91, Zhou17_1, Zhou17_2}. 

The study of the RT instability dates back to the late 1800s, but continues to be a very active area of research. At the beginning of this century, Lafitte et al. \cite{CL00, CCLR01, HL03, Laf01} studied the linear instability of a steady state of inviscid, incompressible fluid flow, using the variational structure of the linearized system. Guo and Hwang \cite{GH03} used the celebrated framework of Grenier \cite{Gre00} to prove the nonlinear RT instability in an unbounded horizontal periodic domain $\bT\times \R$. After that, Lafitte developed equivalent results \cite{Laf08} for a density profile of quasi-isobaric  type. 
For the RT instability problem with boundary conditions,  very recently, Tan and Xu \cite{TX22} extend the result of Guo and Hwang  to the horizontally periodic slab domain $\bT\times (0,h)$ with a positive constant $h$.  

A natural question that builds on the rich history of the RT instability is to ask how a moving domain affects the RT instability in a fluid. In this article, we consider an inviscid, incompressible fluid occupying a moving domain
\[
\Omega(t) = \{(x,t)= (x_1,x_2,t) \in 2\pi L\bT \times \R\times \R_+, -h<x_2<\eta(t,x_1)\}.
\]
Here,  $\bT = \R/\bZ$ is the usual $1D$-torus and $L>0$ is the length of periodicity. We assume that $h > 0$ is the constant depth of the rigid bottom but that the surface function $\eta$ is unknown in the problem. The surface $\Gamma(t) = \{x_2 = \eta(t, x_1)\}$ is the moving upper boundary of $\Omega(t)$, and $\Gamma_h = \{x_2 = -h\}$ is the fixed lower boundary of $\Omega(t)$; see Fiugure \ref{fig:domain}.

\begin{figure}[ht]\label{fig:domain}
 \begin{center}
$
\begin{array}{c}
\includegraphics[width=0.45\textwidth]{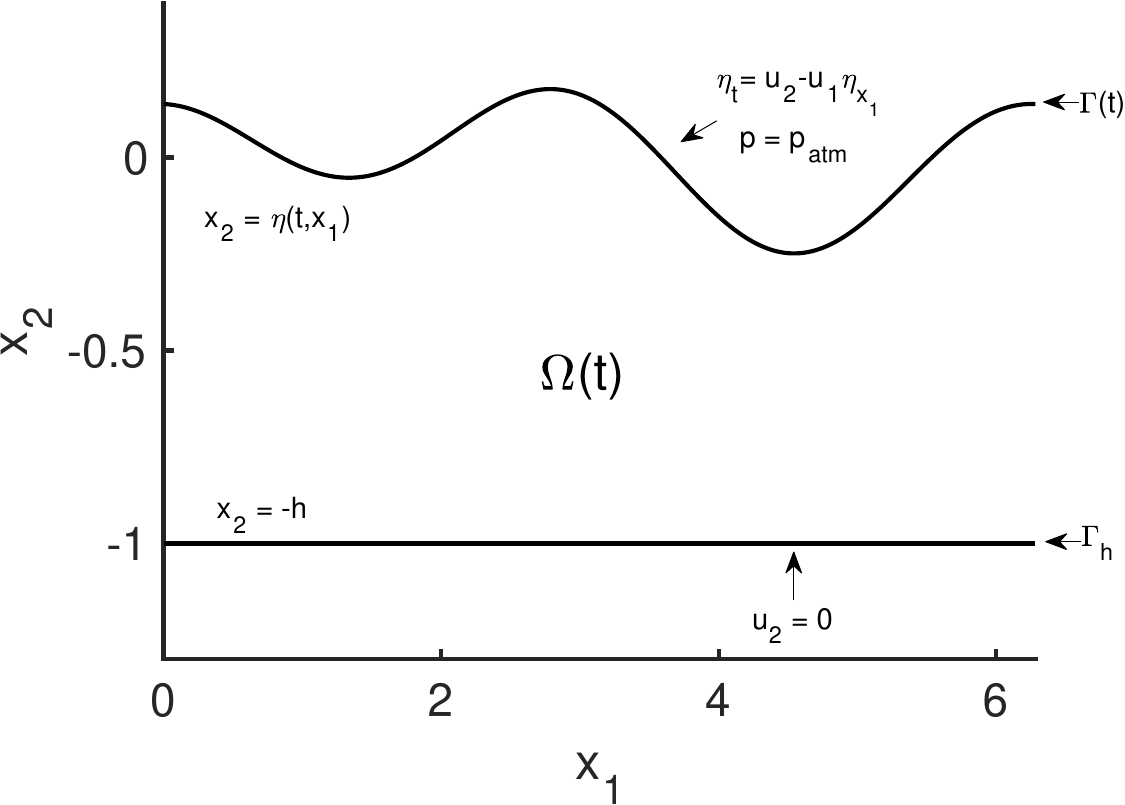}
\end{array}
$
\end{center}
\caption{Depiction of the domain, $\Omega(t)$, for $L = h = 1$.}
\end{figure}

The governing equations are the gravity-driven incompressible Euler equations  in the dimensional form: 
\begin{equation}\label{EqNS}
\begin{cases}
\partial_t \rho +\text{div}(\rho  u) =0,\\
\partial_t(\rho  u) + \text{div}(\rho  u\otimes  u) +\nabla p =- \rho g e_2,\\
\text{div} u=0.
\end{cases}
\end{equation}
 The unknowns $\rho := \rho(x,t)$, $ u := (u_1,u_2)(x,t)$, and $P :=P(x,t)$ represent the density, velocity, and pressure of the fluid, while    $g > 0$ is the gravity constant and $ e_2=(0,1)^T$. On $\Gamma(t)$, we have the following boundary conditions 
\begin{equation}\label{BoundOriginal_1}
\begin{split}
\partial_t \eta = u_2 - u_1\partial_1 \eta,\quad p=p_{atm},
\end{split}
\end{equation}
where $p_{atm}$ is the atmospheric pressure. On $\Gamma_h$, we have the usual non-penetrable boundary condition, which is 
\begin{equation}\label{BoundOriginal_2}
u_2=0.
\end{equation}
For a more physical description of the equations \eqref{EqNS} and the boundary conditions in \eqref{BoundOriginal_1}, we refer to \cite[Sect. 1.8]{Lan13}. 

The movement of the free boundary $\Gamma(t)$ and the subsequent change of the domain $\Omega(t)$ create numerous mathematical difficulties. To circumvent these, as usual, we transform the free boundary problem under consideration \eqref{EqNS}-\eqref{BoundOriginal_1}-\eqref{BoundOriginal_2} to a new problem \eqref{EqNS_Lagrangian} in the fixed domain $\Omega= 2\pi L\bT \times(-h,0)$. The derivation is shown in Section \ref{SectMainResults_1}.

The new, fixed domain problem (Equation \eqref{EqNS_Lagrangian}) admits a steady-state solution of the form
\begin{equation}\label{SteadyState}
(\rho(t,x),  u(t,x), p(t,x), \eta(t,x_1))=(\rho_0(x_2),   0, \bar p(x_2),0 ),
\end{equation}
provided that $\nabla \bar p=-g\rho_0  e_2$ and $\bar p(0)=p_{atm}$.
We are interested in the RT instability of this equilibrium solution, hence, we consider the linear instability of this equilibrium in the case that
\begin{equation}\label{RhoAssume1}
\rho_0 \text{ is strictly increasing and } 0< \rho_- =\rho_0(-h)<\rho_0(0)=\rho_+ <+\infty.
\end{equation} 
 Following Chandrasekhar \cite{Cha61}, establishing linear instability amounts to finding a value of the parameter $\lambda \in \C$  with positive real part such that there exists a solution  $\phi \in H^2(-h,0)$ of the  second-order ordinary differential equation,
\begin{equation}\label{2ndOrderEqPhi}
\lambda^2(k^2 \rho_0 \phi -(\rho_0 \phi')') =gk^2\rho_0' \phi,
\end{equation}
with the boundary conditions
\begin{equation}\label{ODEboundary}
\lambda^2\rho_+ \phi'(0)+gk^2\rho_+\phi(0)=0 \quad \text{and} \quad  \phi(-h)=0.
\end{equation}
Here, $k \in L^{-1}\bZ$ is the transverse wave number and is fixed, and $\phi$ is the vertical component of the perturbed velocity. A value of $\lambda$ for which \eqref{2ndOrderEqPhi} and \eqref{ODEboundary} admit a solution in $H^2(-h,0)$ is a growth rate of the instability, or equivalently, is a characteristic value of the linearized problem (see \cite[Chapter X, Sections 92–93]{Cha61}).

For each fixed value of $k$, there are multiple characteristics. Multiple characteristic values of the linear RT instability were first established by Cherfils and Lafitte \cite{CL00} for a specific profile, as obtained through a precise study of the ODE \eqref{2ndOrderEqPhi} on the real line. In our present setting, we are interested in describing the behavior of the characteristic values of the linearized problem.

Our study is two-fold, examining \eqref{2ndOrderEqPhi}-\eqref{ODEboundary} both theoretically and numerically.   First,  we consider the general case of an increasing profile (satisfying \eqref{RhoAssume1}) and prove the existence of infinitely many characteristic values $\lambda$ to Eq. \eqref{2ndOrderEqPhi}--\eqref{ODEboundary}. We refer to Theorem \ref{ThmGeneral} for the precise statement and to Section \ref{SectGeneral} for the proof. The proof is inspired by an operator method \cite{LN22}, developed by Lafitte and the second author, which uses spectral theory of self-adjoint and compact operators. In the second part, in the spirit of Cherfils and Lafitte \cite{CL00}, we numerically approximate characteristic values of the linear Rayleigh-Taylor problem  by computing zeros of the Evans function, a function whose zeros correspond to characteristic values. The Evans function was introduced by Evans in \cite{Eva1, Eva2, Eva3, Eva4} and since then has undergone much development; for just a few references see \cite{BHLL18,KP,HoZ}.

\section{Derivation of the governed equations and the  main results}\label{SectMainResults}

\subsection{Reformulation in flattening coordinates}\label{SectMainResults_1}

We transform the free boundary problem \eqref{EqNS} into one in the fixed domain $\Omega$ with the two boundaries $\Gamma_0= 2\pi L\bT \times \{0\}$ and $\Gamma_h$ by using the unknown free surface function $\eta$.

Following Beale \cite{Bea81} and then  Guo and Tice \cite{GT13}, we define the Poisson integral for any function $f$ in $2\pi L\bT \times (-\infty,0)$ by
\[
\mathcal Pf(x)= \sum_{ k \in L^{-1}\bZ } \frac{e^{i k x_1}}{\sqrt{2\pi L}}e^{|k|x_2}\hat f(k), \quad\text{where } \hat f(k) =\int_{2\pi L\bT} f(x_1)\frac{e^{-ik x_1}}{\sqrt{2\pi L}} dx_1.
\]
Hence, let $\theta$ be the Poisson extension of $\eta$, we flatten the coordinate domains via the following special coordinate transformation:
\begin{equation}
\Omega \ni x =(x_1,x) \mapsto (x_1, x_2+ \theta(t,x_1,x_2)(1+ \frac{x_2}h)) =: \Theta(t,x)=(y_1,y_2) \in \Omega(t).
\end{equation} 
If $\eta$ is sufficiently small (in an appropriate Sobolev space), then the mapping is a diffeomorphism. This allows us to transform the problem \eqref{EqNS} to one in the fixed spatial domain for each $t>0$. In order to write down the equations in the new coordinate system, we compute
\[
\nabla \Theta  = \begin{pmatrix} 1 &0 \\ (1+\frac{x_2}h)\partial_1\theta  & 1+\partial_2\theta \end{pmatrix}.
\]
and thus define
$ \cA := ((\nabla \Theta)^{-1})^T$. We write the differential operators $\nabla_{\cA}, \text{div}_{\cA}$ with their actions given by
\[
(\nabla_{\cA} f)_i :=  \cA_{ij}\partial_jf, \quad \text{div}_{\cA} X := \sum_{1\leq i,j\leq 2}\cA_{ij}\partial_jX_i.
\]

We now define the density $\tilde \rho$, the velocity ${\tilde  u} $ and the pressure $\tilde p$ on the domain $\Omega$ by the composition 
\[
 (\tilde\rho, {\tilde u},\tilde  p)(t,x) = (\rho,  u,  p)(t,\Theta(t,x)).
\]
We transform \eqref{EqNS}-\eqref{BoundOriginal_1}-\eqref{BoundOriginal_2} into the following system in the new coordinates after dropping the tildes:
\begin{equation}\label{EqNS_Lagrangian}
\begin{cases}
\partial_t\rho - K \partial_t  \theta\partial_2\rho + \text{div}_{\cA}(\rho   u)=0 \quad&\text{in } \Omega, \\
\rho( \partial_t  u - K\partial_t \theta \partial_2  u+ u\cdot \nabla_{\cA} u)+\nabla_{\cA} p = -g\rho  e_2 \quad&\text{in } \Omega, \\
\text{div}_{\cA}  u=0 \quad&\text{in } \Omega, \\
\partial_t\eta = u_2 - u_1\partial_1\eta \quad&\text{on } \Gamma_0, \\
p =p_{atm} \quad&\text{on } \Gamma_0, \\
u_2=0 \quad&\text{on } \Gamma_h.
\end{cases}
\end{equation}

We now rephrase \eqref{EqNS_Lagrangian} in a perturbation formulation around the steady state solution \eqref{SteadyState}.  We define 
\begin{equation}
\zeta = \rho-\rho_0 - \rho_0'\theta, \quad q= p- \bar p+g\rho_0\theta, \quad   u =  u- 0, \quad \theta=  \theta
\end{equation}
The equations for perturbation terms  write
\begin{equation}\label{EqPertur}
\begin{cases}
\partial_t \zeta +\rho_0' u_2= \cQ_1 \quad&\text{in } \Omega,\\
\rho_0\partial_t  u+ \nabla q +g\zeta e_2= \cQ_2 \quad&\text{in }\Omega,\\
\text{div} u= \cQ_3 \quad&\text{in } \Omega,\\
\partial_t \eta - u_2 =\cQ_4 \quad&\text{on } \Gamma_0,\\
q-g\rho_+ \eta=0\quad&\text{on } \Gamma_0,\\
u_2=0 \quad&\text{on }\Gamma_h.
\end{cases}
\end{equation} 
Here $\cQ_1,  \cQ_2, \cQ_3$ and $\cQ_4$ are  nonlinear terms.

\subsection{The linearized equations and main results}
Omitting the nonlinear terms in \eqref{EqPertur}, we obtain the following linearized equations
\begin{equation}\label{EqLinearized}
\begin{cases}
\partial_t \zeta +\rho_0' u_2= 0 \quad&\text{in } \Omega,\\
\rho_0\partial_t  u+ \nabla q +g\zeta  e_2=0  \quad&\text{in }\Omega,\\
\text{div} u= 0 \quad&\text{in } \Omega,\\
\partial_t \eta - u_2 =0\quad&\text{on } \Gamma_0,\\
q-g\rho_+ \eta=0\quad&\text{on } \Gamma_0,\\
u_2=0 \quad&\text{on }\Gamma_h.
\end{cases}
\end{equation}
Following Chandrasekhar \cite{Cha61}, we seek growing normal modes of \eqref{EqLinearized}
\[
(\zeta, u, q)(t,x)= e^{\lambda t}(\omega, v, r)(x), \quad \eta(t,x_1)=e^{\lambda t} \zeta(x_1).
\]
We deduce 
\[
\begin{cases}
\lambda \omega+\rho_0'v_2=0 \quad&\text{in }\Omega,\\
\lambda\rho_0 v+ \nabla r +g\omega  e_2= 0\quad&\text{in }\Omega,\\
\text{div} v =0\quad&\text{in }\Omega,\\
\lambda \zeta= v_2 \quad&\text{on } \Gamma_0,\\
r-g\rho_+\zeta=0 \quad&\text{on } \Gamma_0,\\
v_2=0 \quad&\text{on }\Gamma_h.
\end{cases}
\]
Clearly, we have
\[
\omega = -\frac1{\lambda}\rho_0'v_2, \quad \zeta= \frac1{\lambda}v_2|_{\Gamma_0}
\]
and
\[
\begin{cases}
\lambda^2 \rho_0  v+\lambda \nabla r - g\rho_0' v_2  e_2=0 \quad&\text{in }\Omega,\\
\text{div} v=0 \quad&\text{in }\Omega,\\
\lambda r-g\rho_+v_2= 0 \quad&\text{on } \Gamma_0,\\
v_2=0 \quad&\text{on }\Gamma_h.
\end{cases}
\]
Let $k \in L^{-1}\bZ\setminus \{0\}$, we assume further that
\[
( v,r)(x)=  (\sin(kx_1)\psi(x_2),\cos(kx_1)\phi(x_2), \cos(kx_1)\pi(x_2)).
\]
We deduce that 
\begin{equation}\label{SystPsiPhi}
\begin{cases}
\lambda^2\rho_0 \psi -\lambda k\pi =0 \quad&\text{in } (-h,0),\\
\lambda^2\rho_0\phi + \lambda\pi' =g\rho_0'\phi \quad&\text{in } (-h,0),\\
k\psi+\phi'=0 \quad&\text{in } (-h,0),\\
\lambda\pi(0) -g\rho_+\phi(0) =0, \\
\phi(-h)=0.
\end{cases}
\end{equation}
We obtain from the first and third line of \eqref{SystPsiPhi} that 
\[
\psi= -\frac{\phi'}k, \quad \pi = \frac{\lambda\rho_0 \psi}k= - \frac{\lambda\rho_0 \phi'}{k^2}.
\]
We thus obtain from \eqref{SystPsiPhi} a second-order ODE for $\phi$ in $(-h,0)$, that is \eqref{2ndOrderEqPhi} 
\[
\lambda^2(k^2\rho_0\phi - (\rho_0\phi')')= gk^2\rho_0'\phi,
\]
with the boundary conditions \eqref{ODEboundary}, 
\[
\lambda^2\rho_+ \phi'(0)+gk^2\rho_+\phi(0)=0, \quad \phi(-h)=0.
\]

For any increasing density profile $\rho_0$,  we necessarily have the uniform boundedness of $\lambda$ in $k$, whose proof is based on \cite[Lemma 2.1]{LN22}.
 \begin{lemma}\label{Lemcharacteristic valueReal}
Let $\rho_0$ be increasing, all characteristic values $\lambda$ for Eq. \eqref{2ndOrderEqPhi}-\eqref{ODEboundary} in $H^2(-h,0)$ are always real and satisfy 
\[
\lambda \leqslant \sqrt{\frac{g}{L_0}},
\]
where $L_0:= \big(\max_{(-h,0)} \frac{\rho_0'}{\rho_0}\big)^{-1}$ is the characteristic length of the density profile.
\end{lemma}
In view of Lemma \ref{Lemcharacteristic valueReal}, we only consider positive $\lambda$ in what follows and we look for functions $\phi$ to be real.
 
Our main result is to prove the existence of infinitely many characteristic values to \eqref{2ndOrderEqPhi}-\eqref{ODEboundary}, as stated in Theorem \ref{ThmGeneral}. The proof is given in Section \ref{sec:ThmGeneral} and follows the strategy used in \cite{LN22}.
\begin{theorem}\label{ThmGeneral}
Let $k>0$ and let $\rho_0$ satisfy \eqref{RhoAssume1}.  There exists an infinite  sequence $(\lambda_j, \phi_j)_{j\geq 1}$ with $\lambda_j \in (0,\sqrt{\frac{g}{L_0}})$ and  $\phi_j \in H^2(-h,0)$ satisfying \eqref{2ndOrderEqPhi}-\eqref{ODEboundary}. Furthermore, the sequence $(\lambda_j)_{j\geq 1}$ decreases toward 0 as $j\to \infty$.
\end{theorem}

In Section \ref{sec:numerics}, we focus on two particular density profiles satisfying \eqref{RhoAssume1}, that are a linear function and an exponential function. We derive numerically a quantification of the different growth rates.

\section{The general increasing profile}\label{SectGeneral}

For notational conveniences, we write $x$ instead of $x_2$ in what follows.

\subsection{Auxiliary operators}

We define the frequenlty used function space (for $j=1$ or 2)  
\[
H_\star^j(-h,0) =\{\phi \in H^j(-h,0), \phi(-h)=0\}.
\]
Thanks to the positivity of $\lambda$, we study the following bilinear form.
\begin{proposition}\label{PropPropertyR}
Let us denote by 
\begin{equation}\label{1stBilinearForm}
\begin{split}
\bB_{k,\lambda}(\vartheta, \varrho) & := \frac{gk^2\rho_+}{\lambda^2} \vartheta(0)\varrho(0) + \int_{-h}^0 \rho_0(k^2\vartheta  \varrho + \vartheta'  \varrho') dx.
\end{split} 
\end{equation}
The bilinear form $\bB_{k,\lambda}$ is a continuous and coercive on $H_\star^1(-h,0)$. Furthermore, let  $(H_\star^1(-h,0))'$ be the dual space of $H_\star^1(-h,0)$, which is associated with the norm $\sqrt{\bB_{k,\lambda}(\cdot,\cdot)}$. There exists a unique operator  
\[
Y_{k,\lambda} \in  \mathcal{L}(H_\star^1(-h,0), (H_\star^1(-h,0))'),
\] that is also bijective,  such that
 \begin{equation}\label{EqMathcalB_a}
\bB_{k,\lambda}(\vartheta, \varrho) = \langle Y_{k,\lambda}\vartheta,  \varrho\rangle
\end{equation}
for all $\vartheta, \varrho \in H_\star^1(-h,0)$.
\end{proposition}
\begin{proof}
Clearly, we have the coerciveness of $\bB_{k,\lambda}$ from \eqref{1stBilinearForm}. The existence of $Y_{k,\lambda}$ is due to Riesz's representation theorem. 
\end{proof}

The next proposition describes the properties of $Y_{k,\lambda}$.

\begin{proposition}\label{PropInverseOfR}
We have the following results.
\begin{enumerate}
\item For all $\vartheta \in H_\star^1(-h,0)$, we have
\[
Y_{k,\lambda}\vartheta=k^2 \rho_0 \vartheta -(\rho_0 \vartheta')' \quad\text{ in } \mathcal{D}'(-h,0).
\]  

\item Let $f\in L^2(-h,0)$ be given, there exists a unique solution  $\vartheta \in H_\star^1(-h,0)$ of 
\begin{equation}\label{EqY=f}
Y_{k,\lambda}\vartheta = f \text{ in } ( H_\star^1(-h,0))'.
\end{equation}
Moreover, $\vartheta\in H_\star^2(-h,0)$ and satisfies the boundary conditions \eqref{ODEboundary}.
\end{enumerate}
\end{proposition}
The proof of Proposition \ref{PropInverseOfR} is due to a bootstrap argument, which is followed by \cite[Proposition 3.3]{LN22}. Hence we omit the details here. 

We have the following proposition on $Y_{k,\lambda}^{-1}$, thanks to the continuous injection  from $H^2(-h,0)$ to $L^2(-h,0)$, in the same line of \cite[Proposition 3.4]{LN22}.
\begin{proposition}\label{RemNormR}
The operator $Y_{k,\lambda}^{-1} : L^2(-h,0) \to L^2(-h,0)$ is compact and self-adjoint. 
\end{proposition}

We continue studying the operator $S_{k,\lambda} := \sqrt{\rho'_0} Y_{k,\lambda}^{-1}\sqrt{\rho'_0}$. Owing to Proposition \ref{RemNormR}, we obtain the following. 
\begin{proposition}\label{PropOpeS}
The operator $S_{k,\lambda} : L^2(-h,0) \to L^2(-h,0)$ is compact and self-adjoint.
\end{proposition}

As a result of the spectral theory of compact and self-adjoint operators, the point spectrum of $S_{k,\lambda}$ is discrete, i.e. is a  sequence $\{\gamma_n(k,\lambda)\}_{n\geqslant 1}$ of  characteristic values of $S_{k,\lambda}$, associated with normalized orthogonal eigenfunctions $\{\varpi_n\}_{n\geqslant 1}$ in $L^2(-h,0)$. That means 
\[
\gamma_n(k,\lambda)\varpi_n =S_{k,\lambda}\varpi_n= \sqrt{\rho_0'}Y_{k,\lambda}^{-1}\sqrt{\rho_0'} \varpi_n.
\]
So that with $\phi_n = Y_{k,\lambda}^{-1}\sqrt{\rho_0'} \varpi_n \in H^2(-h,0)$, one has
\begin{equation}\label{EqRf_n}
\gamma_n(\lambda, k) Y_{k,\lambda}\phi_n =  \rho_0' \phi_n
\end{equation}
and $\phi_n$ satisfies \eqref{ODEboundary}. Eq. \eqref{EqRf_n} also tells us that $\gamma_n(k,\lambda) >0$ for all $n$. Indeed, we obtain 
\[
\gamma_n(k,\lambda) \int_{-h}^0 (Y_{k,\lambda}\phi_n)  \phi_n dx = \int_{-h}^0\rho_0'\phi_n^2 dx.
\]
That implies
\begin{equation}\label{EqPhi_nB}
\gamma_n(k,\lambda) \bB_{k,\lambda}(\phi_n,\phi_n) = \int_{-h}^0 \rho_0' \phi_n^2 dx.
\end{equation}
Since $\bB_{k,\lambda}(\phi_n,\phi_n) >0$ and $\rho_0' >0$ on $(-h,0)$, we know that $\gamma_n(k,\lambda)$ is positive for all $n$. Hence, by reordering and using the spectral theory of compact and self-adjoint operators again, we obtain that $\{\gamma_n(k,\lambda)\}_{n\geq 1}$ is a positive sequence decreasing towards 0 as $n\to \infty$.

\subsection{Proof of Theorem \ref{ThmGeneral}}\label{sec:ThmGeneral}

For each $n$,  in order to verify that $\phi_n$ is a solution of \eqref{2ndOrderEqPhi}-\eqref{ODEboundary}, we are left to look for real values of $\lambda_n$  such that 
\begin{equation}\label{EqDefineLambda}
\gamma_n(k,\lambda_n)=\frac{\lambda_n^2}{gk^2}.
\end{equation}
To solve \eqref{EqDefineLambda}, we need the following  three lemmas.
\begin{lemma}\label{LemGammaCont}
For each $n$, $\gamma_n(k,\lambda)$ and $\phi_n$ are differentiable in $\lambda$.
\end{lemma}
The proof of Lemma \ref{LemGammaCont} is the same as \cite[Lemma 3.2]{LN22}, we omit the details here.
\begin{lemma}\label{LemMax}
There holds
\begin{equation}\label{quotient}
\max_{\phi \in H_\star^1(-h,0)} \frac{k^2 \phi^2(0)}{\int_{-h}^0 (k^2\phi^2 +(\phi')^2)dx}=\frac{\tanh(kh)}k.
\end{equation}
\end{lemma}
\begin{proof}
Let us consider the Lagrangian functional 
\[
L(\phi,\beta)= k^2 \phi^2(0) -\beta \Big(\int_{-h}^0 (k^2\phi^2 +(\phi')^2)dx-1\Big).
\]
for any $\phi \in H_\star^1(-h,0)$. Using the Lagrange multiplier theorem, we find that the extrema of the quotient \eqref{quotient} are necessarily the stationary points  $(\phi_\star,\beta_\star)$ of $L$, which satisfy  
\[
\int_{-h}^0  (k^2\phi_\star^2 +(\phi_\star')^2)dx=1
\]
and
\[
k^2\phi_\star(0)\theta(0)= \beta_\star \int_{-h}^0 (k^2\phi_\star \theta+\phi_\star'\theta')dx,
\]
for any $\theta\in H_\star^1(-h,0)$. Using integration by parts, we obtain  that 
\begin{equation}\label{1Eq_LemMax}
k^2\phi_\star(0)\theta(0) -\beta_\star \phi_\star'(0)\theta(0) = \beta_\star \int_{-h}^0 (k^2\phi_\star-\phi_\star'') \theta dx.
\end{equation}
Restricting $\theta \in C_0^{\infty}(-h,0)$, the resulting equality yields $\phi_\star''-k^2\phi_\star =0$ on  $(-h,0)$.
That implies  $\phi_\star(x)= A \sinh(k(x+h))$ for some constant $A$ satisfying that 
\[
A^2 \int_{-h}^0 k^2 (\sinh^2(k(x+h))+\cosh^2(k(x+h)))dx=1.
\]
Inserting that $\phi_\star$ into \eqref{1Eq_LemMax}, we have
\[
A (k^2 \sinh(kh)- \beta_\star k \cosh(kh))=0 \rightarrow \beta_\star =  \frac{\tanh(kh)}k.
\]
This completes the proof of Lemma \ref{LemMax}.
\end{proof}
\begin{lemma}\label{LemGammaDecrease}
For each $n$, the following holds:
\begin{itemize}
    \item[1.] $\gamma_n(k,\lambda)$ is strictly increasing as $\lambda>0$,
    \item[2.] $\frac{\lambda^\alpha}{\gamma_n(k,\lambda)}$ is strictly increasing in the regime $\lambda >\epsilon>0$, with 
    \[
    \alpha \geq 2>\frac{2g\rho_+ k}{g\rho_+ k +  \rho_- \epsilon^2\tanh(kh)}.
    \]
\end{itemize}
\end{lemma}
\begin{proof}
Let $z_n= \frac{d\phi_n}{d\lambda}$, it follows from  \eqref{EqRf_n} that 
\begin{equation}\label{1stEqDeriTz_n}
\lambda^\alpha Y_{k,\lambda}z_n +\alpha \lambda^{\alpha-1} Y_{k,\lambda}\phi_n = \frac{\lambda^\alpha}{\gamma_n(k,\lambda)} \rho_0'z_n+ \frac{d}{d\lambda}\Big( \frac{\lambda^\alpha}{\gamma_n(k,\lambda)}\Big)\rho_0'\phi_n
\end{equation}
in $(-h,0)$. Note that 
\begin{equation}\label{ZnAtMinusA}
z_n(-h)=0, \quad \lambda^2\rho_+ z_n'(0)+ gk^2\rho_+ z_n(0) -\frac{2gk^2\rho_+}{\lambda} \phi_n(0)=0.
\end{equation}
Hence, multiplying by $\phi_n$ on both sides of \eqref{1stEqDeriTz_n}, we obtain that 
\begin{equation}\label{2ndEqDeriTz_n}
\begin{split}
&\lambda^\alpha \int_{-h}^0 (Y_{k,\lambda}z_n) \phi_n dx+ \alpha \lambda^{\alpha-1} \int_{-h}^0 (Y_{k,\lambda}\phi_n)\phi_n dx\\
&\qquad = \frac{\lambda^\alpha}{\gamma_n(k,\lambda)} \int_{-h}^0 \rho_0'z_n \phi_n dx+ \frac{d}{d\lambda}\Big( \frac{\lambda^\alpha}{\gamma_n(k,\lambda)}\Big) \int_{-h}^0\rho_0' \phi_n^2 dx.
\end{split}
\end{equation}
Using the integration by parts and also \eqref{ZnAtMinusA}, we have 
\begin{equation}\label{2ndEqDeriTz_nTz_n}
\begin{split}
\int_{-h}^0 (Y_{k,\lambda}z_n) \phi_n dx&=  \int_{-h}^0 (Y_{k,\lambda} \phi_n)z_n dx -\rho_0 z_n' \phi_n \Big|_{-h}^0 + \rho_0 \phi_n'z_n\Big|_{-h}^0\\
&= \int_{-h}^0 (Y_{k,\lambda} \phi_n)z_n dx -\frac2{\lambda^3} gk^2 \rho_+ (\phi_n(0))^2.
\end{split}
\end{equation}
We deduce from \eqref{2ndEqDeriTz_n} that
\[
\begin{split}
\frac{d}{d\lambda}& \Big(\frac{\lambda^\alpha}{\gamma_n(k,\lambda)}\Big) \int_{-h}^0\rho_0'\phi_n^2dx \\
&= -\frac2{\lambda^{3-\alpha}} gk^2 \rho_+ (\phi_n(0))^2 + \alpha \lambda^{\alpha-1} \bB_{k,\lambda}(\phi_n,\phi_n) \\
&= \lambda^{\alpha-1} \Big( - \frac{(2-\alpha)}{\lambda^2} gk^2 \rho_+ (\phi_n(0))^2 + \alpha \int_{-h}^0 \rho_0(k^2\phi_n^2+ (\phi_n')^2)dx \Big).
\end{split}
\]

The first assertion is proven by letting $\alpha=0$.
To prove the second assertion, we estimate that
\[\begin{split}
&- \frac{(2-\alpha)}{\lambda^2} gk^2 \rho_+ (\phi_n(0))^2+ \alpha \int_{-h}^0 \rho_0(k^2\phi_n^2+ (\phi_n')^2)dx  \\
&\qquad\qquad \geq \Big(  \frac{\alpha \rho_-\tanh(kh)}k  - \frac{(2-\alpha)g\rho_+}{\epsilon^2} \Big) k^2(\phi_n(0))^2>0,
\end{split}\]
if 
\[
 \frac{\alpha \rho_-\tanh(kh)}k   - \frac{(2-\alpha)}{\epsilon^2} g\rho_+ >0 \quad\Leftrightarrow \quad 
\alpha > \frac{2 g\rho_+ k}{g\rho_+ k + \rho_- \epsilon^2 \tanh(kh) }.
\]
This ends the proof of Lemma \ref{LemGammaDecrease}.
\end{proof}
Now we are in the situation to solve \eqref{EqDefineLambda} and  to complete the proof of Theorem \ref{ThmGeneral}.
\begin{proof}[Proof of Theorem \ref{ThmGeneral}]
Using \eqref{EqPhi_nB}, we know that
\[
\frac1{\gamma_n(k,\lambda)} \int_{-h}^0 \rho_0' \phi_n^2 dx = \int_{-h}^0 (Y_{k,\lambda}\phi_n)\phi_n dx = \bB_{k,\lambda}(\phi_n,\phi_n), 
\]
that implies, for all $n\geqslant 1$,
\begin{equation}\label{LimitGammaRight}
\lim_{\lambda \to \sqrt{\frac{g}{L_0}}} \frac{\lambda^2}{\gamma_n(k,\lambda)} > gk^2.
\end{equation}
Next, we prove that 
\begin{equation}\label{LimitGammaLeft}
\lim_{\lambda \to 0} \frac{\lambda}{\gamma_n(k,\lambda)} =0.
\end{equation}
For $0<\lambda \leq \epsilon$, due to Lemma \ref{LemGammaDecrease}(2), we have that  
\[
\frac{\lambda^\alpha}{\gamma_n(k,\lambda)} \leq \frac{\epsilon^\alpha}{\gamma_n(k,\epsilon)} \quad\text{with } \alpha = \frac{2 g\rho_+ k}{g\rho_+ k + \rho_- \lambda^2 \tanh(kh) } <2.
\]
That implies
\[
\frac{\lambda^2}{\gamma_n(k,\lambda)}= \lambda^{2-\alpha} \frac{\lambda^\alpha}{\gamma_n(\lambda,k)} \leq \lambda^{2-\alpha} \frac{\epsilon^\alpha}{\gamma_n(k,\epsilon)} \leq \lambda^{2-\alpha} \frac{\epsilon^2}{\gamma_n(k,\epsilon)} \to 0 \quad\text{as }\lambda \to 0,
\]
yielding \eqref{LimitGammaLeft}. Combining the two limits \eqref{LimitGammaRight} and \eqref{LimitGammaLeft} with Lemma \ref{LemGammaDecrease}(1), we deduce that, for each $n\geq 1$, there is a unique $\lambda_n$ solving \eqref{EqDefineLambda}. Since $\lambda_n$ is a characteristic value, we have that $\lambda_n\in (0,\sqrt{\frac{g}{L_0}})$.

We  prove that the sequence $(\lambda_n)_{n\geq 1}$ is  decreasing. Indeed, if $\lambda_m<\lambda_{m+1}$ for some $m\geq 1$, we make use of Lemma \ref{LemGammaDecrease}(2) to have $\frac{\lambda_m^2}{\gamma_m(k,\lambda_m)}< \frac{\lambda_{m+1}^2}{\gamma_m(k,\lambda_{m+1})}$.
Note that  $\gamma_m(k,\lambda_{m+1}) > \gamma_{m+1}(k,\lambda_{m+1})$, hence
\[
gk^2= \frac{\lambda_m^2}{\gamma_m(k,\lambda_m)} <  \frac{\lambda_{m+1}^2}{\gamma_{m+1}(k,\lambda_{m+1})} =gk^2.
\]
That contradiction tells us that $(\lambda_n)_{n\geq 1}$ is a decreasing sequence.  

To complete the proof of Theorem \ref{ThmGeneral},  we now prove that $\lim_{n\to \infty}\lambda_n =0$. Indeed,  suppose that $\lim_{n\to \infty}\lambda_n=c_0>0$, one has that $\lambda_n\geq c_0$ for all $n\geq 1$. By Lemma \ref{LemGammaDecrease}(2) again, it yields
\[
\frac{c_0^2}{\gamma_n(k,c_0)} \leq \frac{\lambda_n^2}{\gamma_n(k,\lambda_n)} =\frac1{gk^2}.
\]
As a result, we have
\[
c_0^2\leq \frac{\gamma_n(k,c_0)}{gk^2}.
\]
Letting $n\to \infty$, we obtain a contradiction that $c_0\leq 0$. Hence, $\lim_{n\to \infty}\lambda_n =0$.
 The proof of Theorem \ref{ThmGeneral} is complete.
\end{proof}

\section{Numerical simulation}\label{sec:numerics}

In this section, we describe our shooting‐method algorithm for computing the Rayleigh–Taylor growth‐rate spectrum for a linear  density profile
\begin{equation}\label{RhoLinear}
\rho_0(x) = \rho_- + \frac{\rho_+-\rho_-}{h} (x+h).
\end{equation}
and for an exponential density profile  
\begin{equation}\label{RhoExp}
\rho_0(x)= \rho_- e^{\frac1h \ln(\frac{\rho_+}{\rho_-}) (x+h)}.
\end{equation}
For each transverse wavenumber $k>0$, we compute nontrivial solutions of \eqref{2ndOrderEqPhi} (subject to \eqref{ODEboundary}) using a \emph{shooting method} that yields the discrete spectrum $\{\lambda_j(k) \}_{j\geq 1}$.

\subsection{Shooting procedure}
To provide a shooting procedure, we rewrite the ODE \eqref{2ndOrderEqPhi}  as the first order system 
\begin{equation}\label{system-y12}
\begin{cases}
\phi' = \psi,\\[6pt]
\psi' = -\dfrac{\rho_0'}{\rho_0} \psi + k^2\phi   - \dfrac{gk^2}{\lambda^2}\dfrac{\rho_0'}{\rho_0}\phi.
\end{cases}
\end{equation}
Hence, we define the surface-residual
\begin{equation}\label{Evans}
E(\lambda) = \lambda^2\rho_+ \psi(0)+gk^2\rho_+\phi(0).
\end{equation}
A characteristic value $\lambda$ satisfies $E(\lambda)=0$. A function like $E(\lambda)$ that is constructed from a basis of ODE solutions and whose zeros correspond to eigenvalues of a linearized PDE system is often called an Evans function, see \cite{BHLL18, KP}.
The following table summarizes the algorithm.
\begin{center}
\fbox{
\begin{minipage}{0.95\textwidth}

\textbf{Algorithm: Shooting–Evans Method}

\medskip

\textbf{Input:} $g,\, k,\, \rho_-,\, \rho_+,\, h,$ 
search interval $[\lambda_{\min}, \lambda_{\max}]$

\textbf{Output:} Discrete spectrum $\{\lambda_j(k)\}$

\medskip

Initialize density profile $\rho_0(x)$.
Construct a grid $\{\lambda_i\} \subset [\lambda_{\min}, \lambda_{\max}]$.

\medskip

\textbf{for each candidate value $\lambda$ do}

\quad Rewrite the second–order ODE as the first–order system:
\[
\begin{cases}
y_1' = y_2, \\
y_2' =
- \dfrac{\rho_0'}{\rho_0} y_2
+ k^2 y_1
- \dfrac{gk^2}{\lambda^2}
\dfrac{\rho_0'}{\rho_0} y_1
\end{cases}
\]

\quad Impose bottom boundary condition:
\[
y_1(-h)=0, \qquad y_2(-h)=1
\]

\quad Integrate system on $x \in [-h,0]$.

\quad Extract $\phi(0)=y_1(0)$ and $\phi'(0)=y_2(0)$.

\quad Compute Evans function:
\[
E(\lambda)
=
\lambda^2 \rho_+ \phi'(0)
+
g k^2 \rho_+ \phi(0)
\]

\textbf{end for}

\medskip

Detect sign changes of $E(\lambda)$.

\textbf{for each sign-change interval do}

\quad Apply root–finding method (bisection/Brent).

\quad Compute refined root $\lambda_j$.

\textbf{end for}

\medskip

Return all eigenvalues $\{\lambda_j(k)\}$.

\end{minipage}
}
\end{center}

\subsection{Linear density profile}

For the linear profile \eqref{RhoLinear}, let us fix
\begin{equation}\label{parameters}
\rho_- = 0.1,\quad \rho_+ = 0.2,\quad h = 2, \quad L = 1,\quad g = 9.8.
\end{equation}
Hence, we have that $\sqrt{\frac{g}{L_0}}\approx 2.21$ and the search interval is $[0,2.21]$. 

For $k=680$, Figure \ref{fig:first5bindings} shows the first five characteristic values $\lambda_j$ $(1\leq j\leq 5)$ plotted with black dots. As shown in Figure \ref{fig:first5bindings}, the top characteristic values are strictly decreasing $\lambda_1>\dots>\lambda_5$ and are bounded above by $\sqrt{\frac{g}{L_0}}\approx 2.21$, which is indicated by a dashed blue line.
\begin{figure}[ht]
  \centering
  \includegraphics[width=0.7\linewidth]{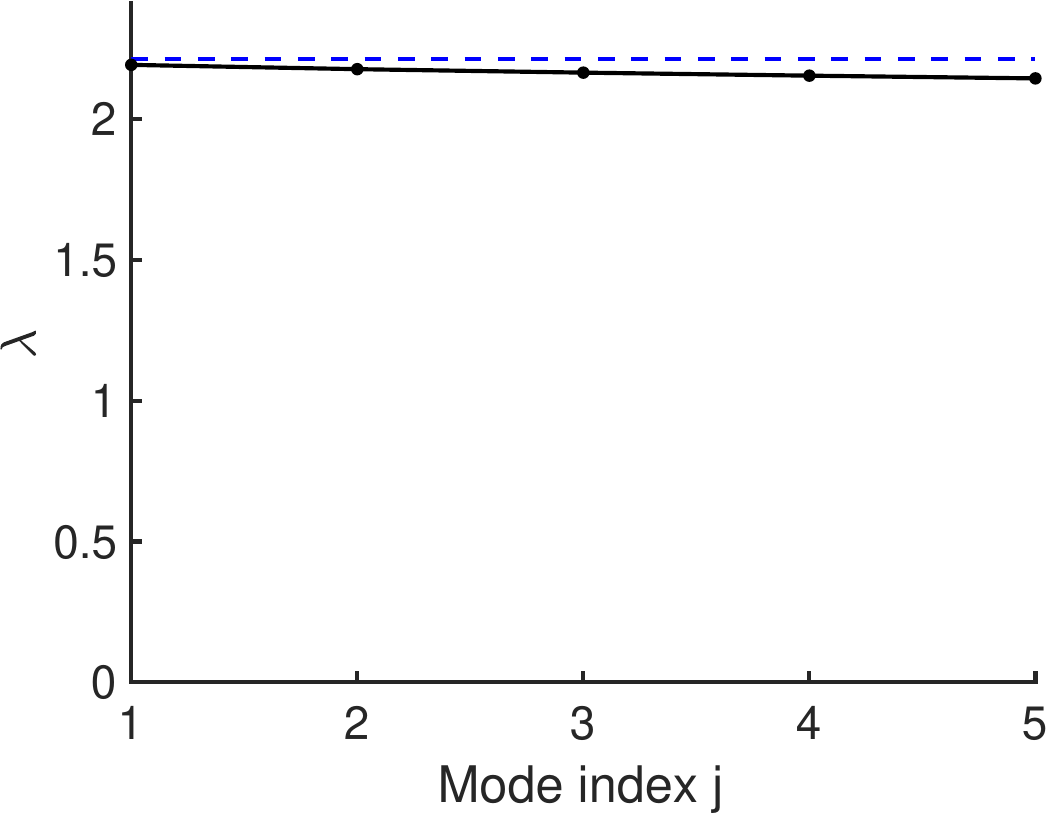}
  \caption[
    First five modes and uniform bound for the linear profile
  ]{Computed characteristic values \(\lambda_j\) for \(j=1,\dots,5\) when $k = 680$.}
  \label{fig:first5bindings}
\end{figure}

In Figure \ref{fig:linear:fs}(a), we plot the top characteristic mode $\lambda_1$, and in Figure \ref{fig:linear:fs}(b), we plot the second top characteristic $\lambda_2$. 

\begin{figure}[ht]
 \begin{center}
$
\begin{array}{lr}
(a) \includegraphics[scale=0.35]{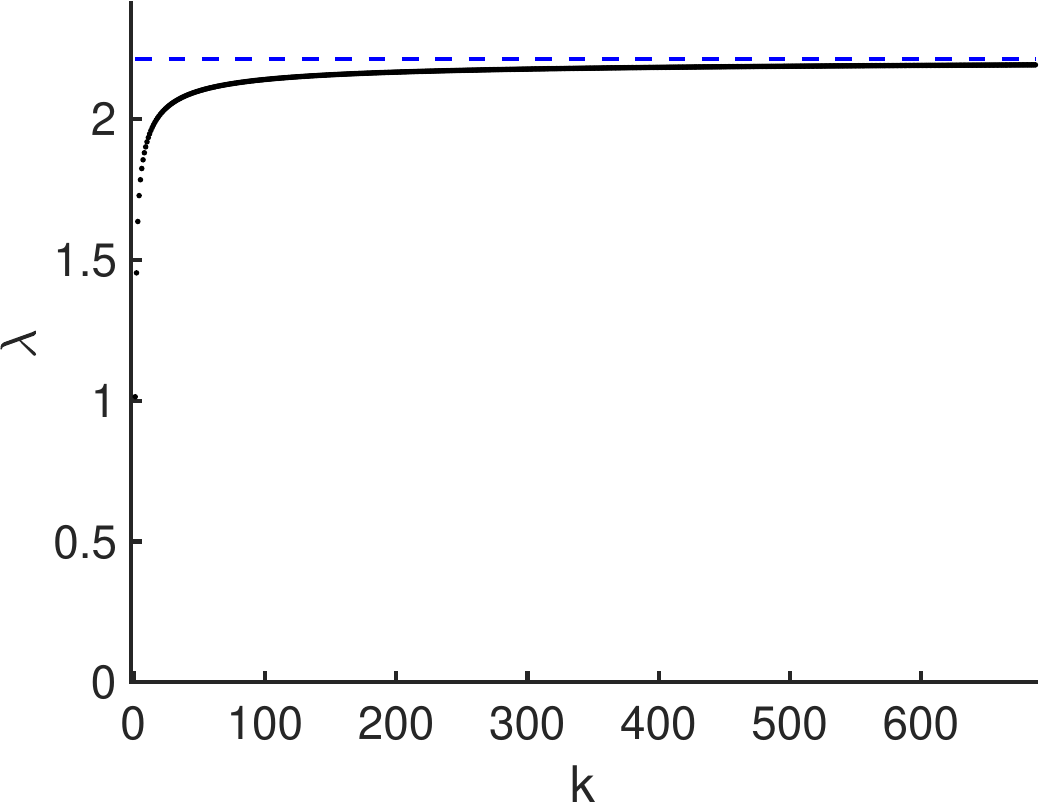} & (b) \includegraphics[scale=0.35]{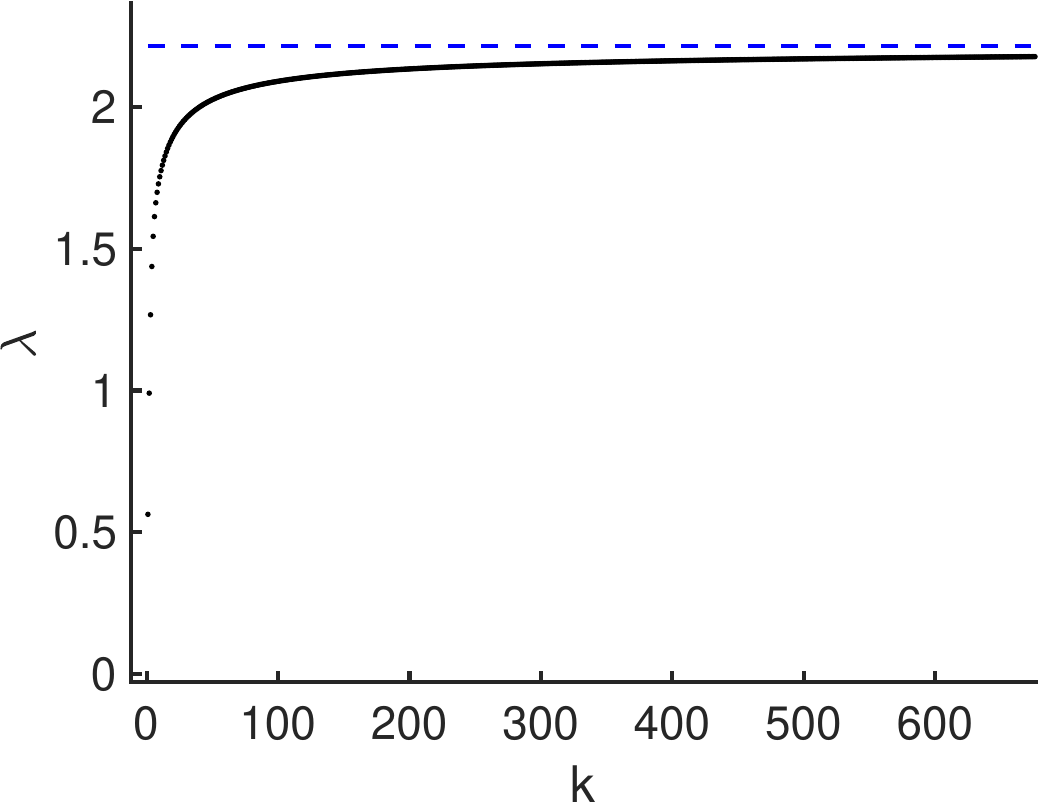}
\end{array}
$
\end{center}
\caption{(a) Plot of $\lambda_1(k)$ against $k$ (black dots) and the upper bound on eigenvalues (blue dashed line). (b) Plot of $\lambda_2(k)$ against $k$ (black dots) and the upper bound on eigenvalues (blue dashed line). The parameters are $\rho_- = 0.1$, $\rho_+ = 0.2$, $h = 2$, $L = 1$, and the Atwood number is approximately $0.3466$.}
\label{fig:linear:fs}
\end{figure}



All the numerical experiments confirm both the infinite descending chain  $\lambda_j(k) \searrow 0$ for fixed $k$ and the uniform bound $\sqrt{g/L_0}$ for large $k$.

\subsection{Exponential density profile} In this section, we consider an exponential density profile of the form 
\[
\rho_0(x) = \rho_-e^{a(x+h)},
\]
where $a\in \R_+$ is such that $\rho_0(0) = \rho_+$. Consequently, $a = \frac1h \ln(\rho_+/\rho_-)$ and 
\[
\rho_0(x) = \rho_+ e^{\ln(\rho_+/\rho_-)\frac{x}h}=  \rho_+\left(\frac{\rho_+}{\rho_-}\right)^{x/h},
\]
that is \eqref{RhoExp}.

Now $\frac{\rho_0'(x)}{\rho_0(x)} = a=\frac{\ln(\rho_+/\rho_-)}{h}.$ Thus the ODE to solve for computing the Evans function is given by
\eq{
y_1'&= y_2,\\
y_2'&= k^2\left(1-\frac{ag}{\lambda^2}\right)y_1-ay_2,
}{\notag}
with initial conditions $y_1(-h) = 0$ and $y_2(-h) = 1$. 

We define $\mu:= \sqrt{a^2+4k^2(1-\frac{ag}{\lambda^2})}$. If $\mu = 0$, then the system to solve is
\eq{
\mat{y_1\\y_2}'&= \mat{0&1\\- \frac14 a^2&-a}\mat{y_1\\y_2}.
}{\notag}
If $\lambda$ is a characteristic value for $\mu =0 $, then $E(\lambda) =0$ implies that $\lambda^2 = -gk^2y_1(0)/y_2(0)$ and from the defintion of $\mu$, that 
\[
0 = a^2+4k^2(1-\frac{ag}{\lambda^2})=a^2+4k^2+4a \frac{y_2(0)}{y_1(0)}
\]
as seen via substitution for $\lambda^2$. Then 
\[
k^2 = -a \frac{y_2(0)}{y_1(0)}-\frac14 a^2.
\]
We numerically verify for our choice of parameters that $-ay_2(0)/y_1-a^2/4 \approx -0.4343$. Thus, for our choice of parameters, there are no characteristic values for $\mu = 0$ since $k$ must be real valued.

If $\mu \neq 0$, then the solution to this ODE is given by
\eq{\mat{y_1(x)\\y_2(x)} &= e^{\frac12(-a+\mu)(x+h)}\mat{\frac1\mu \\ -\frac{a}{2\mu}+\frac12} +e^{-\frac12(a+\mu)(x+h)}\mat{-\frac1\mu\\ \frac{a}{2\mu}+\frac12}.}{\notag}
Thus the Evans function \eqref{Evans} is given by
\eq{
E(\lambda) &= \rho_+(\lambda^2y_2(0)+gk^2y_1(0))\\
&= \frac{\rho_+e^{-\frac12 ah}}{\mu}\left[ \lambda^2\left( -a\sinh(\frac12 \mu h)+\mu \cosh(\frac12 \mu h)\right)+2gk^2\sinh(\frac12 \mu h)\right]. 
}{\notag}
By substituting 
\[
\lambda^2 = \frac{4k^2ag}{a^2+4k^2-\mu^2},
\]
which is derived from the definition of $\mu$, and rearranging terms, we find that
\eq{E(\lambda) = \frac{2g\rho_+k^2e^{-ah/2}}{\mu(a^2+4k^2-\mu^2)}\left[ (4k^2-a^2-\mu^2) \sinh(\frac12 \mu h)+2a\mu \cosh(\frac12 \mu h) \right].}{\notag}

If $E(\lambda) = 0$, then 
\eq{(4k^2-a^2-\mu^2)\sinh(\frac12\mu h)+2a\mu \cosh(\frac12 \mu h) = 0,}{\notag} or alternatively, 
\eq{(4k^2-a^2-\mu^2)\tanh(\frac12 \mu h) + 2a\mu = 0.}{\notag} 
If $\mu >0$, then
\eq{
\lambda^2 = -\frac{2gk^2\tanh(\mu h/2)}{\mu-a\tanh(\mu h/2)}.
}{\notag}
The denominator is zero if $\mu = 0$, and the derivative of the denominator with respect to $\mu$ is always greater than $\frac12(2-ah) = \frac12 (2- \ln(\rho_+/\rho_-))$. For the choice of $\rho_+$ and $\rho_-$ used in our numerical study, this quantity is positive, which implies $\lambda$ is purely imaginary. Hence, we conclude $\mu\not >0$.

Thus, we may conclude that $\mu$ is purely imaginary, say $\mu = i\xi$ with $\xi>0$. Then we have that 
\eq{
(4k^2-a^2+\xi^2)\tan(\frac12 \xi h)+2a\xi = 0.
}{\notag}
We may solve for $k$ in terms of $\xi$ to obtain 
\eq{k(\xi)=\frac{1}{2}\sqrt{\frac{a^2-\xi^2}{4}-\frac{1}{2}a\xi\cot(\frac12 \xi h)}.}{\notag}
Then 
\eq{
\lambda(\xi) = \sqrt{\frac{4k^2(\xi)ag}{a^2+4k^2(\xi)+\xi^2}}.
}{\notag}
Thus, $(\xi,k(\xi),\lambda(\xi))$ is a curve parameterized by $\xi>0$ that passes through the eigenvalues $\lambda$ corresponding to $k$; see Figure \ref{fig:exponential:parameterized} (a), which was plotted using the online plotting tool Desmos. To solve for these eigenvalues for fixed $k$, we use a bisection method to find each root of 
\[
g(\xi):= (4k^2-a^2+\xi^2)\tan(\frac12\xi h)+2a\xi
\]
that lies between poles of $\tan$, that is for $-\frac{\pi}h+\frac{2\pi j}h<\xi< \frac{\pi}h+\frac{2\pi j}h$, with  $j\in \mathbb{N}$. There is only one root between each pole. Indeed, for $k\neq 0$, 
\eq{
\cos^2(\frac12 h\xi)g'(\xi)
&= 2\xi\sin(\frac12  h\xi )\cos(\frac12  h\xi)+\frac12 h(4k^2-a^2+\xi^2)+a\cos^2(\frac12  h\xi)\\ &\geq 2h-a^2+\frac12 h\xi^2-2\xi \\
&= \frac{h}{2}\left[\left(\xi-\frac{1}{2h}\right)^2-\frac{1}{4h^2}+\frac{2}{h}(2h-a^2)\right].}{\notag}
For our choice of parameters, $-\frac{1}{4h^2}+\frac{2}{h}(2h-a^2)\approx 0.74$, thus $g'(\xi)>0$ and there is only one root of $g(\xi)$ between each of its poles. 

In Figure \ref{fig:exponential:fs}(a), we plot the top characteristic mode $\lambda_1$, and in Figure \ref{fig:exponential:fs}(b), we plot the second top characteristic $\lambda_2$. For $k=200$, Figure \ref{fig:exponential:parameterized} (b) shows the first five characteristic values $\lambda_j$ $(1\leq j\leq 5)$ plotted with black dots. As shown in Figure \ref{fig:exponential:parameterized}(b), the top characteristic values are strictly decreasing $\lambda_1>\dots>\lambda_5$ and are bounded above by $\sqrt{\frac{g}{L_0}}=\sqrt{ag}\approx 3.8786$, which is indicated by a dashed blue line.

\begin{figure}[ht]
 \begin{center}
$
\begin{array}{lr}
(a) \includegraphics[scale=0.27]{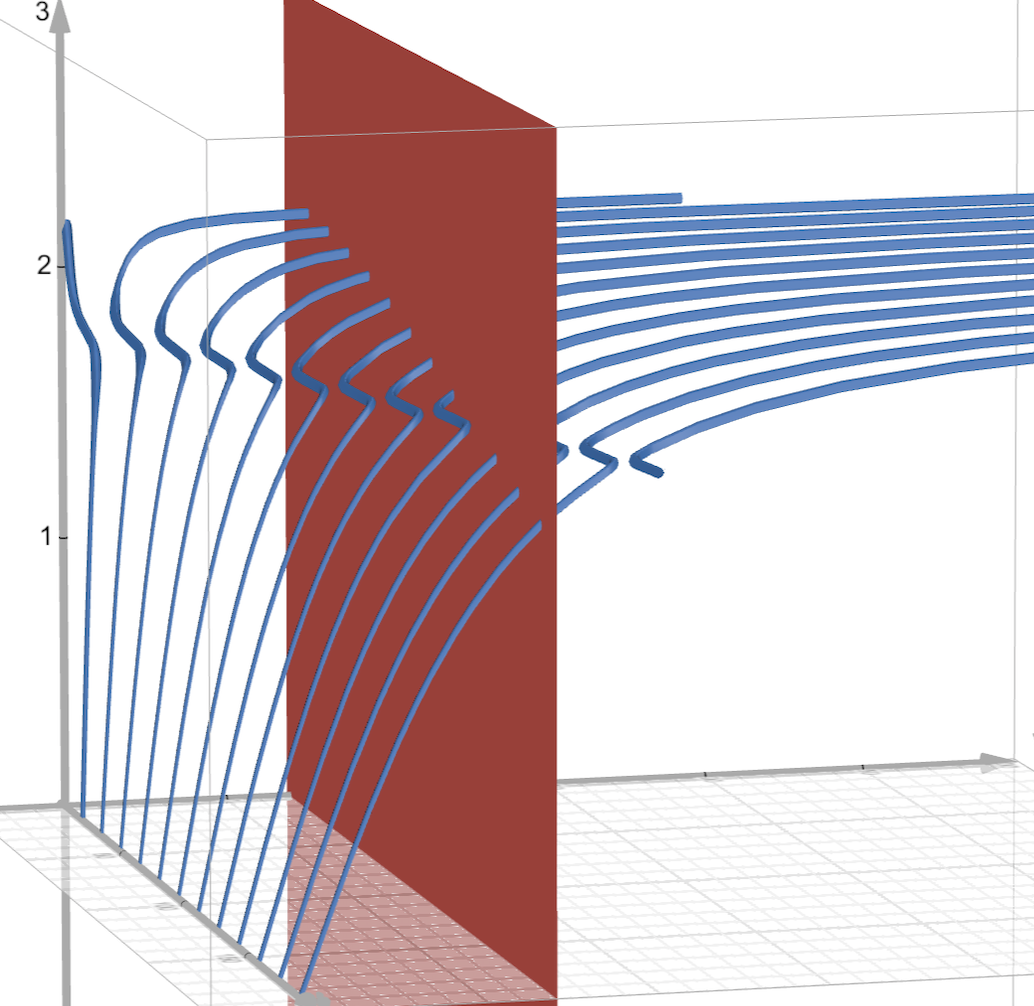} & (b) \includegraphics[scale=0.35]{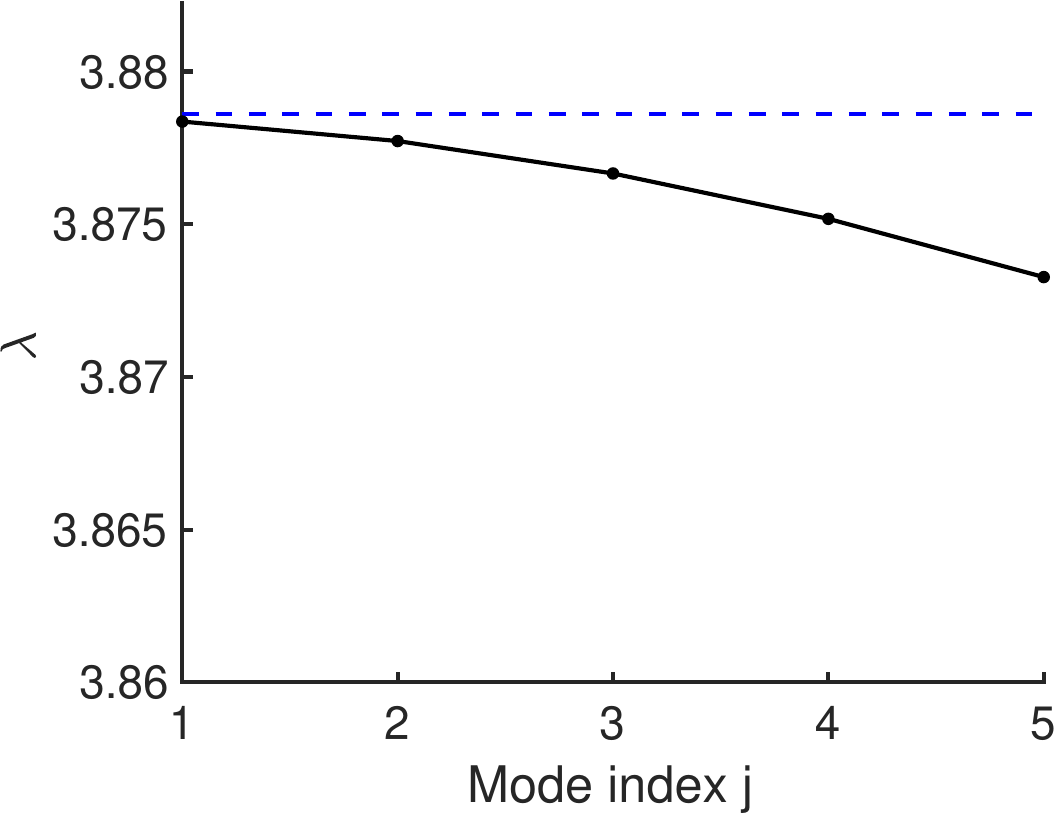}
\end{array}
$
\end{center}
\caption{(a) Plot of $(\xi,k(\xi),\lambda(\xi))$ with a blue line. A red plane corresponds to a specific value of $k$ and the intersection of the blue curve with the red plane indicates the eigenvalues for that value of $k$. (b) Computed characteristic values \(\lambda_j\) for \(j=1,\dots,5\) when $k = 200$. }
\label{fig:exponential:parameterized}
\end{figure}

\begin{figure}[ht]
 \begin{center}
$
\begin{array}{lr}
(a) \includegraphics[scale=0.35]{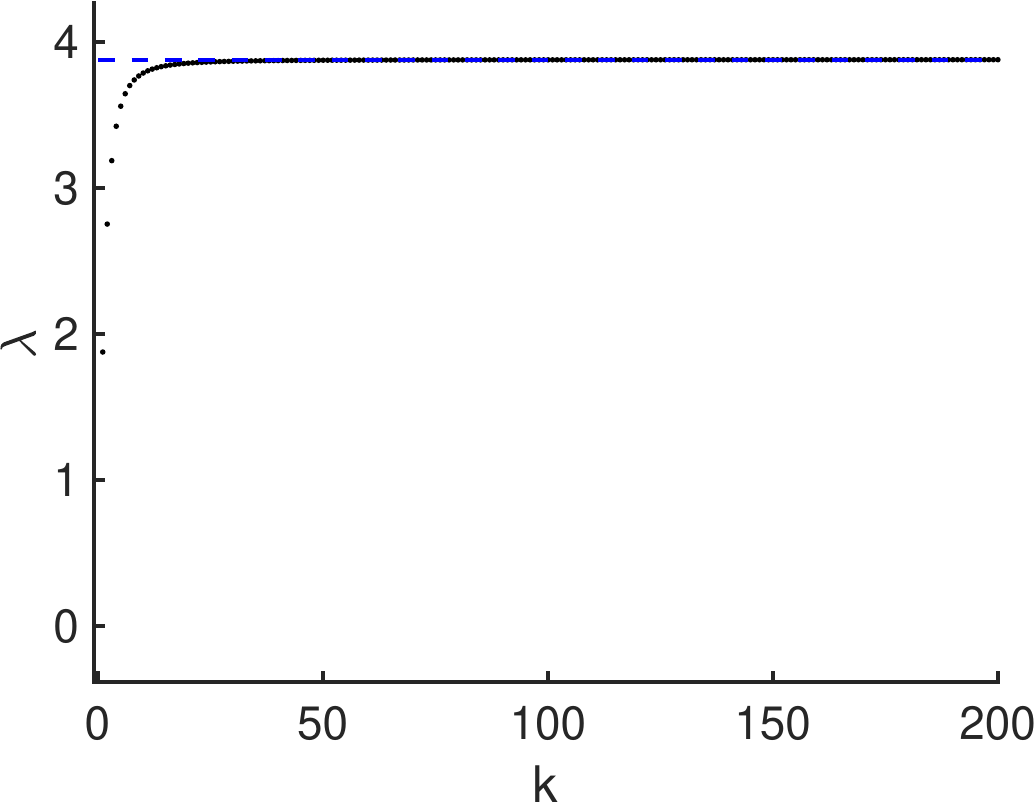} & (b) \includegraphics[scale=0.35]{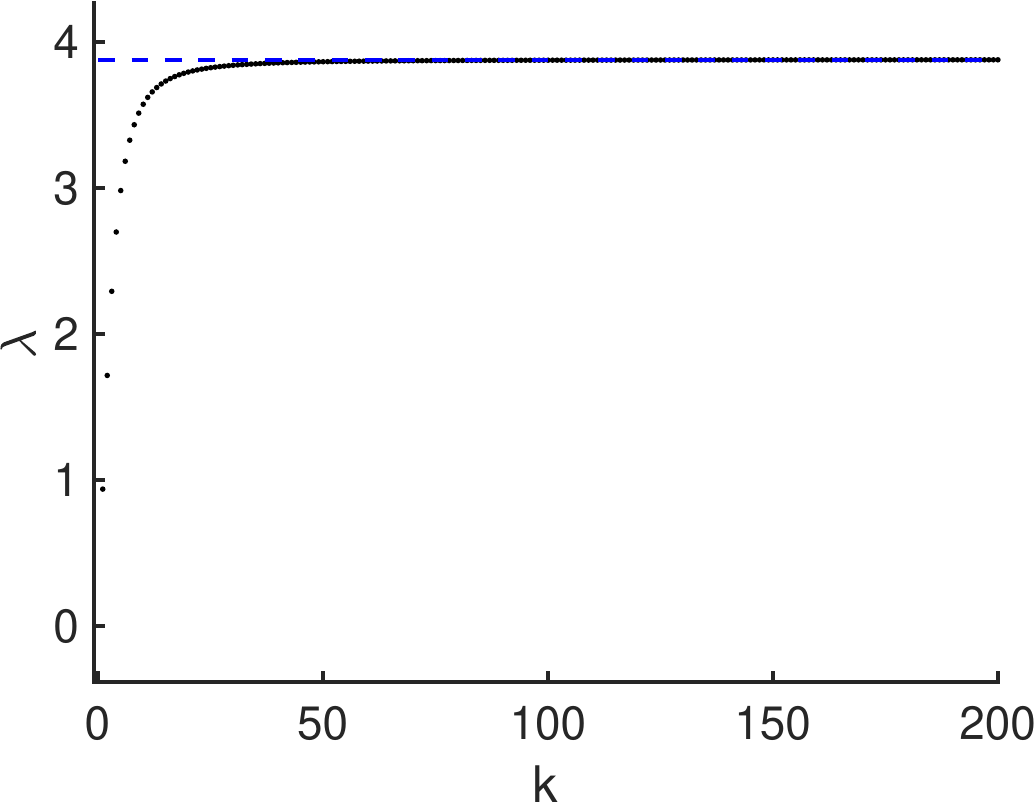}
\end{array}
$
\end{center}
\caption{(a) Plot of $\lambda_1(k)$ against $k$ (black dots) and the upper bound on eigenvalues (blue dashed line). (b) Plot of $\lambda_2(k)$ against $k$ (black dots) and the upper bound on eigenvalues (blue dashed line). The parameters are $\rho_- = 0.1$, $\rho_+ = 1$, $h = 1.5$, $L = 1$, $a \approx 1.535$, and the Atwood number is $90/99$.}
\label{fig:exponential:fs}
\end{figure}

\subsection{Code}
Figure \ref{fig:exponential:parameterized}(a) was plotted using Desmos. The code for creating all the other figures is available at the Nonlinear Waves repository on GitHub; see \cite{stablab}. The Python code was translated from the Matlab code with the aid of ChatGPT.



\clearpage
\begin{thebibliography} {999999}

\bibitem{BHLL18}
\textsc{B. Barker, J. Humpherys, J. Lytle, G. Lyng}, Evans function computation for the stability of travelling waves, \textit{Phil. Trans. R. Soc. A} \textbf{376}: 20170184.



\bibitem{Bea81}
\textsc{J. T. Beale}, The initial value problem for the Navier–Stokes equations with a free surface, \textit{Commun. Pure Appl. Math.} \textbf{34} (1981) 359--392.

\bibitem{Cha61}
\textsc{S. Chandrasekhar}, Hydrodynamics and Hydromagnetic Stability, Oxford University Press, London, 1961.




\bibitem{CL00} 
\textsc{C. Cherfils and O. Lafitte}, Analytic solutions of the Rayleigh equation for linear density profiles \textit{Phys. Rev. E} \textbf{62}, 2967 (2000).

\bibitem{CCLR01}
\textsc{C. Cherfils-Cl\'erouin, O. Lafitte, P.-A. Raviart}, Asymptotic results for the linear stage of the Rayleigh-Taylor instability. In  \textit{Adv. Math. Fluid Mech.}, Birkh\"auser, Basel, 47–71 (2001).


\bibitem{Eva1}
\textsc{J. W. Evans},  Nerve axon equations I: Linear approximations, \textit{Indiana Univ. Math. J.}, \textbf{21} 1971/72, pp. 877--955.

\bibitem{Eva2}
\textsc{J. W. Evans}, Nerve axon equations II:Stability at rest, \textit{Indiana Univ. Math. J.} \textbf{22} 1972/73, pp. 75--90.

\bibitem{Eva3}
\textsc{J. W. Evans},  Nerve axon equations III: Stability of the nerve impulse, \textit{Indiana Univ. Math. J.} \textbf{22} 1972/73, pp. 577--593.

\bibitem{Eva4}
\textsc{J. W. Evans}, Nerve axon equations IV: The stable and the unstable impulse, \textit{Indiana Univ. Math. J.} \textbf{24} 1974/75, pp. 1169--1190.


\bibitem{GH03}
\textsc{Y. Guo, H. J. Hwang},
On the dynamical Rayleigh--Taylor instability, \textit{Arch. Rational Mech. Anal.} \textbf{167} (2003), pp. 235--253.


\bibitem{GT13}
\textsc{Y. Guo, I. Tice}, Local well-posedness of the viscous surface wave problem without surface tension, \textit{Anal. PDE} \textbf{6}:2 (2013), 287--369.

\bibitem{Gre00}
\textsc{E. Grenier}, On the nonlinear instability of Euler and Prandtl equations, \textit{Commun. Pure Appl. Math.} \textbf{ 53}, (2000), pp.  1067--1091.


\bibitem{HL03}
\textsc{B. Helffer,  O. Lafitte}, Asymptotic methods for the characteristic values of the Rayleigh equation for the linearized Rayleigh-Taylor instability, \textit{ Asymptotic Analysis} \textbf{ 33} (2003), pp. 189--235.


\bibitem{Kull91}
\textsc{H. Kull}, Theory of the Rayleigh-Taylor instability, \textit{Phys. Rep.} \textbf{206} (1991), pp. 197--325.


\bibitem{Lan13}
\textsc{D. Lannes}, \textit{The water waves problem: Mathematical analysis and asymptotics}, Mathematical survey
and monographs, vol. 188, 2013,  AMS, Providence.



\bibitem{Laf01}
\textsc{O. Lafitte}, Sur la phase lin\'eaire de l'instabilit\'e de Rayleigh-Taylor. S\'eminaire Equations aux D\'eriv\'ees Partielles du Centre de Math\'ematiques de l'Ecole Polytechnique, Ann\'ee 2000–2001.

\bibitem{Laf08}
\textsc{O. Lafitte}, The linear and nonlinear Rayleigh-Taylor instability for the quasi-isobaric profile, \textit{Phys. D} \textbf{237} (2008), pp. 1602--1639.

\bibitem{LN22}
\textsc{O. Lafitte, T.-T. Nguyen}, Spectral analysis of the incompressible viscous Rayleigh--Taylor system, \textit{Water Waves} \textbf{4} (2022), pp. 259--305.

\bibitem{Lin98}
\textsc{J. D. Lindl}, 1998. Inertial Confinement Fusion. Springer.






\bibitem{Rem00}
\textsc{B. A. Remington, R. P. Drake, H. Takabe, D. Arnett},  A review of astrophysics experiments on intense lasers, \textit{Phys. Plasmas} \textbf{7} (2000), pp. 1641--1652.


\bibitem{Str83}
\textsc{J.W. Strutt (Lord Rayleigh)}, Investigation of the character of the equilibrium of an incompressible heavy fluid of variable density, \textit{Proc. London Math. Soc.} \textbf{14} (1883), pp. 170--177.


\bibitem{TX22}
\textsc{Z. Tan, S. Xu}, The Rayleigh--Taylor instability of incompressible Euler equations in a horizontal slab domain, \textit{J. Differential Equations} \textbf{319} (2022), pp. 100--130.

\bibitem{Tay50}
\textsc{G. Taylor}, The instability of liquid surfaces when accelerated in a direction perpendicular to their planes, \textit{Proc. R. Soc. Lond. Ser. A} \textbf{201} (1950), pp. 192--196.


\bibitem{Zhou17_1}
\textsc{Y. Zhou},  Rayleigh–Taylor and Richtmyer–Meshkov instability induced flow, turbulence, and mixing. I., \textit{Phys. Rep.} \textbf{720–722} (2017), pp. 1--136.

\bibitem{Zhou17_2}
\textsc{Y. Zhou}, Rayleigh–Taylor and Richtmyer–Meshkov instability induced flow, turbulence, and mixing. II., \textit{Phys. Rep.} \textbf{723–725} (2017), pp. 1--160.

\bibitem{KP}
    \textsc{Todd Kapitula and Keith Promislow}, Spectral and dynamical stability of nonlinear waves, \textit{Springer} (2013).

\bibitem{HoZ}
	\textsc{Howard, Peter and Zumbrun, Kevin}, The {E}vans function and stability criteria for degenerate viscous shock waves, \textit{Discrete and Continuous Dynamical Systems. Series A.} (2004), \textbf{Vol 10}, pp. 837--855.

\bibitem{stablab}
\textsc{Barker, Blake and Humpherys, Jeffrey and Lytle, Joshua and Zumbrun, Kevin}, {STABLAB}: A {MATLAB}-based numerical library for Evans function computation, 
\textit{https://github.com/nonlinear-waves/stablab.git}

    

\end{thebibliography}
\end{document}